\documentclass[12pt]{amsart}
\usepackage{graphicx, times}
\pdfoutput=1

\DeclareMathOperator{\Card}{Card}
\let\mc=\mathcal

\newcommand{\R}{\ensuremath{\mathbb R}}
\newcommand{\Z}{\ensuremath{\mathbb Z}}
\newcommand{\N}{\ensuremath{\mathbb N}}
\newcommand{\zd}{\ensuremath{\mathbb Z_{12}}}
\newcommand{\zc}{\ensuremath{\mathbb Z_{c}}}
\newcommand{\fa}{{\mathcal F}_{A}}
\newcommand{\fb}{{\mathcal F}_{B}}
\newtheorem{Thm}{Theorem}
\newtheorem{Prop}{Proposition}

\newtheorem*{Rem}{Remark}
\newtheorem{Lemma}{Lemma}

\newcommand{\dessin}[4]{
\begin{figure}[h]
\centerline{\includegraphics[width =#2]{#1}}
\caption{#3}
\label{#4}
\end{figure}}

\title{About the number of generators of a musical scale.}

\begin{document}
\author{Emmanuel Amiot}\thanks{manu.amiot@free.fr}
\date{5/21/2009, CPGE, Perpignan, France }

\maketitle

{\em Abstract\\ A finite arithmetic sequence of real numbers has exactly two generators: the sets
$\{a, a+f, a+2f,\dots a+ (n-1)f = b\}$ and
$\{b, b-f, b-2f, \dots b-(n-1)f = a\}$ are identical. 
A different situation exists when dealing with arithmetic sequences modulo some integer $c$. 
The question arises in music theory, where it has long been acknowledged that certain musical scales are generated, i.e. are arithmetic sequences modulo the octave. It is easy to construct many scales whose number of generators is a given totient number. We prove in this paper that no other number of generators can arise, and a complete classification is given.\medskip
 
In other words, starting from musical scale theory, we answer the mathematical question of how many different arithmetic sequences in a cyclic group share the same support set.
Some extensions and generalizations are also provided, notably for arithmetic sequences of real numbers modulo 1.}\medskip

\bigskip

\section*{Foreword}
In the cyclic group $\zc$, arithmetic sequences are not ordered as they are in $\Z$. For instance, the sequence (0 7 14 21 28 35 42) reduces modulo 12 to (0 7 2 9 4 11 6), whose support is (rearranged) $\{0,2,4,6,7,9,11\} \subset \zd$.
In music theory, when $\zc$ is taken as a model of an equal chromatic universe with
 $c$ pitch classes modulo the octave, an interval $f$ generates several scales, and conversely many musical scales can be constructed as the values of an arithmetic sequence with step $f$. In the example above,  the G major  scale is generated by $f=7$, using the convention where 0=C, 1=C$\sharp$,2=D\dots 11=B. Identification modulo the octave, i.e. modulo 12 arithmetic, equals the third note 0+7+7 to 14-12=2, i.e. D, and so on. All 12 major scales can be generated in that way. \medskip
 
 Musicians say that the major scale is generated by ascending fifths (mathematically meaning steps of $f=7$) or descending fourths (meaning $-f=5$ which is of course the same thing modulo 12).
For example, the C major scale is the set $\{0, 2, 4, 5, 7, 9, 11\}$ which can be produced as both sequences
(5 0 7 2 9 4 11) or (11 4 9 2 7 0 5).
Similarly, the whole-tone scale  $\{0, 2, 4, 6, 8, 10\}$ is generated by either $f=2$ or $12-f=10$, starting the sequence with any element of the scale.\medskip

It looks natural to infer that, as in $\Z$, such arithmetic sequences have exactly two possible (and opposite) generators. But such is not the case.
\begin{enumerate}
\item
The `almost full' scale, with $c-1$ notes,
is generated by any interval $f$ coprime with $c$: $f, 2f, \dots (d-1)f$ mod $c$, 
being $c-1$ different elements, builds up such a scale\footnote{ 
 This was mentioned to Norman Carey by Mark Wooldridge \cite{Carey}, chap. 3.}. 
 For $c=12$ for instance we have $\Phi(12) = 4$ different generators, $\Phi$ being {Euler}'s totient function.\footnote{ $\Phi(n)$ is usually defined as the number of generators of the cyclic group with $n$ elements. For $n=12$ these generators are 1,5,7 and 11.} Same thing for the full aggregate, i.e. $\zc$ itself.
 \item
 Another extreme case is the `one note scale', wherein any number can be viewed as a generator as the scale is reduced to its starting note.\footnote{We will leave aside the even wilder case of an empty scale with no note at all.}
 \item
When $f$ generates a subgroup of $\zc$, then $k f$ generates the same scale as $f$,
for any $k$ coprime\footnote{ This case was suggested by David Clampitt
in a private communication; it also appears in \cite{Quinn}.} with $c$, that is
to say the generators are all those elements of the group $(\zc, +)$ whose order is 
equal to some particular divisor of $c$.
Consider for instance a `whole-tone scale' in 14 tone equal temperament, i.e. a 7-element sequence generated by 2 modulo 14. It exhibits 6 generators, which are the elements of $\Z_{14}$ with order 7, namely the even numbers: the set $\{0,2,4,6,8,10,12\}$ is produced by either of the 6 following sequences
\begin{center}
(0 2 4 6 8 10 12), (0 4 8 12 2 6 10) , (0 6 12 4 10 2 8), (0 8 2 10 4 12 6),
 (0 10 6 2 12 8 4), (0 12 10 8 6 4 2).
\end{center}
Mathematically speaking, the subgroup of $\zc$ with $d$ elements where $d$ divides $c$, is generated by precisely $\Phi(d)$ intervals. 
Conversely, the subgroup of $\zc$ generated by $f$ is the (one and only)  cyclic subgroup with $c/\gcd(c,f))$ elements. 
\item
One last example: 
the `incomplete whole-note scale' in ten notes equal temperament $\{1, 3, 5, 7\}$ has 4 different 
generators, namely 2, 4, 6, 8 (with as many different starting points).
\end{enumerate}

The present paper studies all possible arithmetic sequences (aka generated scales) 
$\{a, a+f, a+2f, \dots a+(d-1)f\}$ in $\zc$. We will prove that the above examples cover all possible cases, and hence that the number of generators is always a `totient number', i.e. some $\Phi(n)$ where $\Phi$ is Euler's totient function.\footnote{ Sloane's integer sequence A000010.}
\begin{itemize}
\item 
When a generator $f$ is coprime with $c$, there are two generators only, except 
in the cases of the full and `almost full scale' (when $d=c$ or $d=c-1$ which admit $\Phi(c)$ generators.
\item 
When some generator of the scale is not coprime with $c$, it will be seen that the number 
of generators can be arbitrarily large. 
\item
A related result involving complementation will be stated. 
\item 
 Then we will endeavour to bring the question into a broader focus, considering partial periodicity and its  relationship to Discrete Fourier Transform.
\item
Lastly, for a generalization, some results on scales with non integer generators will be given.
\end{itemize}

\subsection*{Notations and conventions}\ 

Unless otherwise mentioned, computations take place in $\zc$, the cyclic group with
$c$ elements. 

$\Z$ (respectively $\N=\Z^+$) stands for the integers (respectively the non negative integers).

$a\mid b$ means that $a$ is a divisor of $b$ in the ring of integers.

Since this paper may be of interest to music theorists, the word `scale' is used, incorrectly but according to custom, for `pc-set', i.e. an unordered subset of \zc.

A `generated scale' is a subset of $\zc$ build from the values of
some finite arithmetic sequence (modulo $c$), e.g.
$A = \{a, a+f, a+2f \dots\}\subset \zc$. 

`ME set' stands for Maximally Even Set',
`WF' means `Well Formed', `DFT' is `Discrete Fourier Transform'.

$\Phi$ is Euler's totient function, i.e. $\Phi(n)$ is the number of integers smaller than $n$
and coprime with $n$.

We will say that $A\subset \zc$ is generated by $f$ if $A$ can be written as
$A = \{a, a+f, a+2f \dots \}$ for some suitable starting point $a\in A$. Notice that sets are denoted using curly brackets, while sequences are given between parenthesis.

\section{Results}
\subsection{The simpler case}
  For any generated scale \emph{where the generator $f$ is coprime
with $c$}, there are only two generators $f$ and $c-f$, except for the extreme cases mentioned 
in the foreword:

\begin{Thm}\label{2genCase}\ \\
   Let $1<d<c-1$; the scales $A=\{0,a,2a,\dots  (d-1)a\}$ and  $B=\{0, b, 2b,\dots  (d-1)b\}$
   with $d$ pitch classes, generated in $\zc$ by intervals $a,b$, where one at least of $a,b$ is {coprime with $c$}, cannot coincide up to translation, unless $a=b$ or $a+b=c$ (i.e. $b=-a \mod c$). In other words, $A$ admits only the two generators $a$ and $-a$.
   
  But when $d=c-1$ or $d=c$, then there are $\Phi(c)$ generators.
\end{Thm}
In simpler words, this means that when some generating interval of the scale is coprime with $c$, then we are dealing with the simplest case, the case of the major scale which is generated by fifths or fourths exclusively -- except if the scale almost saturates $\zc$. 

\subsection{About polygons of $\zc$}
\begin{Thm}\label{polygon}
   A {\em regular polygon} in \zc, i.e. a translate of some subgroup $f\zc$ with $d$ elements,
   has exactly $\Phi(d)$ generators.
\end{Thm}
Of course, for any such generator, any point of the polygon can be used as starting point
for the generation of the scale.

 \dessin{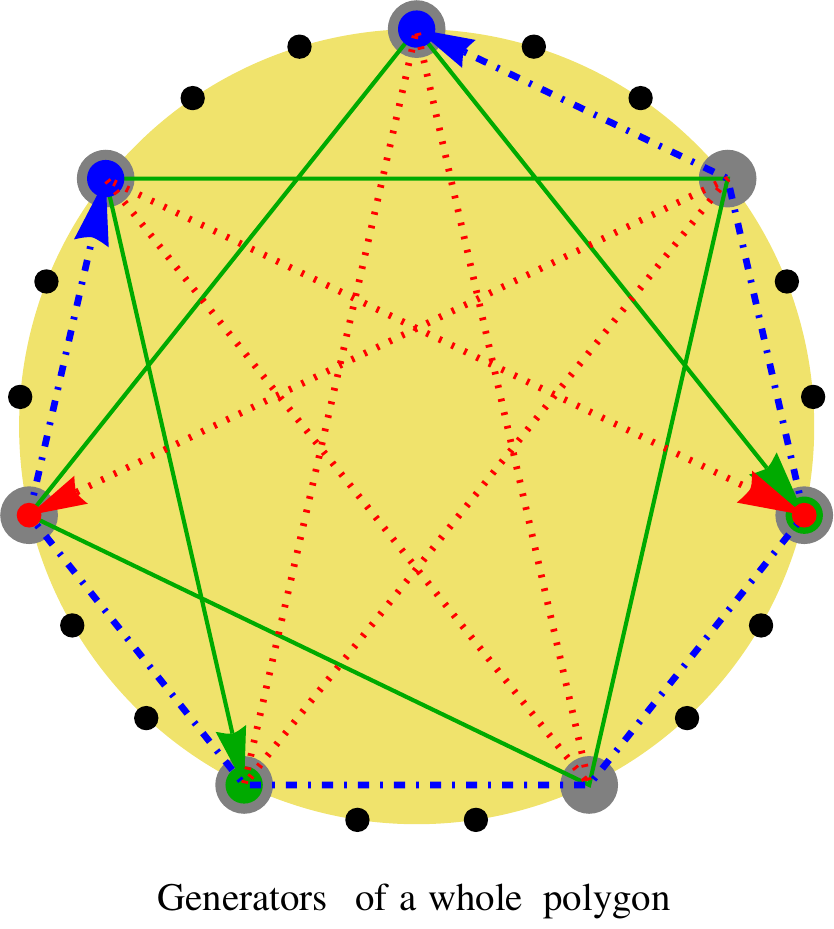}{6 cm}{Many generators for a regular polygon}{fullPoly}

Clearly in that case $f$ will not be coprime with $c$ (except in the extreme case of $d=c$ when the polygon is the whole aggregate $\zc$). The following theorem rounds up the classification of all possible cases:

\begin{Thm}\label{notCoprime}
   A scale generated by $f$ \textbf{not} coprime with $c$, with a cardinality
   $1<d<c$, has 
   \begin{itemize}   
      \item 
       one generator when the scale is (a translate of) $\{0, c/2\}$ (what musicians call a tritone);
      \item
     two generators (not coprime with $c$) when $d$ is strictly between 1 and $c'-1=c/m - 1$ where
    $m = \gcd(c, f)$;
        \item
    $\Phi(d)$ generators when $d = c'=c/m$, i.e. when $A$ is a regular polygon;
      \item
   $\Phi(d)$ generators when $d = c'-1$, and these generators share the same order in the group $(\zc,+)$;
   \end{itemize}
\end{Thm}
As we can see on fig. \ref{4generators}, the scales in Thm. \ref{notCoprime} are 
\textsl{geometrically} not new, since they are enlarged versions of the previous cases
(those of Thm. 1) immersed  in larger chromatic universes.

The last, new case, features {\em incomplete regular polygons}, i.e. regular polygons 
with one point removed. For any divisor $c'$ of $ c$, let $f=c/c'$ and $d=c'-1$: 
$$\{f, 2f, \dots d\, f \pmod c\} \qquad\text{ where  } d\,f = c - f = -f$$
is the simplest representation of such a scale.
$\{1, 3, 5, 7\}$ in a ten-note universe, mentioned in the foreword, was one of them.
Fig. \ref{4generators} shows another one in $\Z_{16}$.

 \dessin{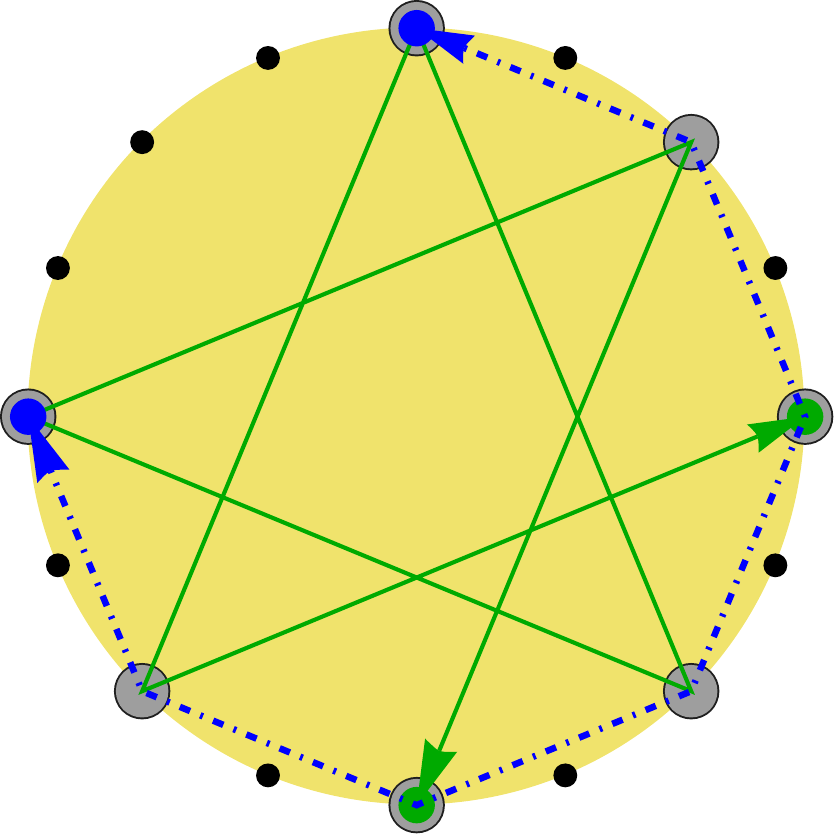}{5.5 cm}{Four generators for this scale}{4generators}

Consideration of the `gap' in such scales yields an amusing fact, 
apparent in this example and whose proof will be left to the reader:
\begin{Prop}
    For each different generator of an `incomplete polygon', there is a different starting point of the arithmetic sequence.
\end{Prop}
For instance, scale $\{0, 2, 4, 6, 8, 10, 12\}$ in $\Z_{16}$ has the $\Phi(8)=4$ generators
2, 6, 10, 14 with starting points 0, 4, 8, 12 respectively. See fig. \ref{4generators}, noticing that both
paths can be reversed.

Summing up, the number of generators of a scale can be any number in the set of images (or codomain) of Euler's totient function $\Phi$, which is made of 1 and most of the even integers.
Conversely, {\em nontotient numbers}, that is to say numbers that are not a $\Phi(n)$, can never be the number of generators of a scale. All odd numbers (apart from 1) are nontotient, the sequence of even nontotients begins with 14, 26, 34, 38, 50, 62, 68, 74, 76, 86, 90, 94, 98\dots \footnote{ Sloane's sequence A005277 in his online encyclopedy of integer sequences. For the whole sequence including odd numbers, see A007617.} So 14 is the smallest even number than can never be the number of different generators of a scale.

 Leaving aside the extreme cases of one-note scales and tritones, 
 the geometry of generated scales comes in three types:
\begin{itemize}
  \item 
  Regular polygons,
\item 
  Regular polygons minus one note,
  \item
  The seminal case: `major scale-like scales', i.e. scales with only two (opposite) generators.  
\end{itemize}
So this seminal case is by no means the only one. The three cases are summarized in picture \ref{3cases}.

 \dessin{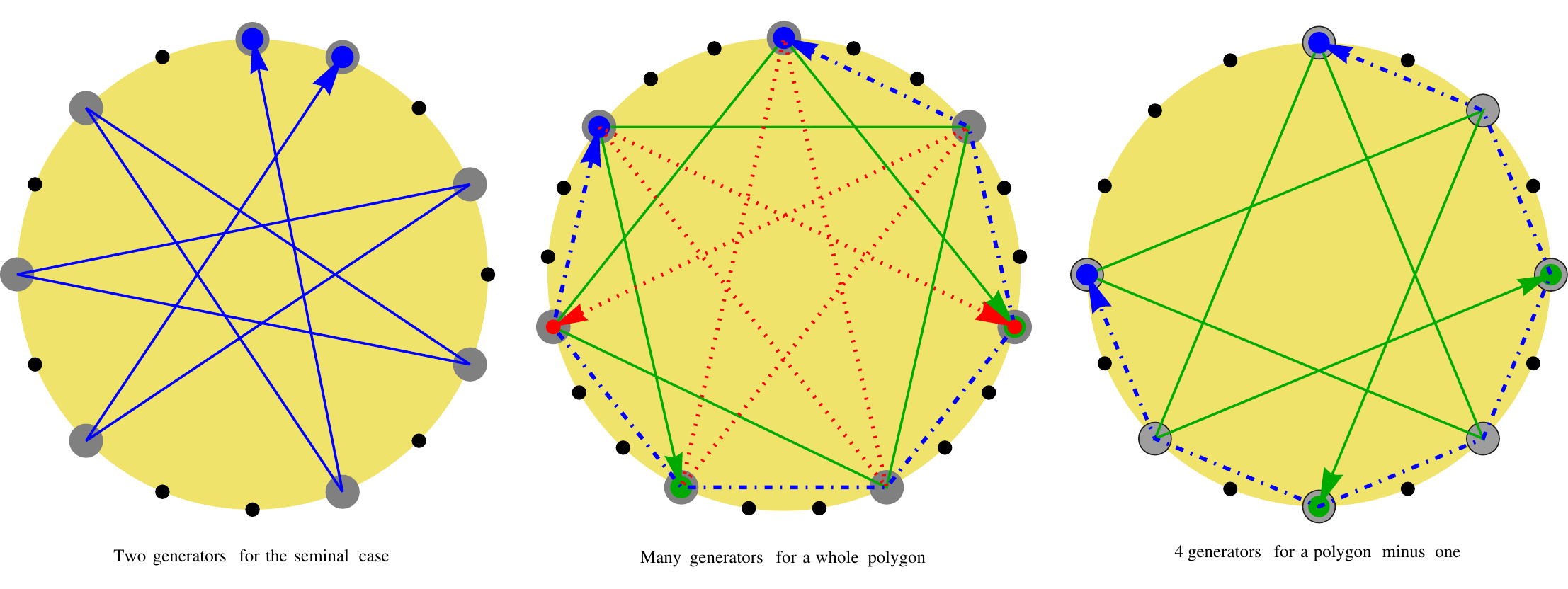}{17 cm}{Three cases}{3cases}

\subsection{Partial periodicities and Fourier Transform}
In the case of two generators, the values of the generators can be conveniently recognized as the indexes of the maximum Fourier coefficients of the scale. Informally, a generated scale $A$ is at least partially periodic, and this is apparent on the Fourier transform $\fa:t\mapsto \sum_{k\in A} e^{-2i\pi\, k\, t/c}$. 

More precisely, in the online supplementary of \cite{Amiot} we proved the following:
   \begin{Thm} \label{MEcoprime}
        For $c,d$ coprime, a scale with $d$ notes is generated by an interval $f$,
        \underline{coprime with $c$}, if and only if the semi-norm
        $$
          {\| \fa \|}^* = \max_{t \text{ coprime with }c} |\fa(t)|
          = \max_{t \text{ coprime with }c} \bigl|\sum_{k\in A} e^{-2i\pi\, k\, t/c}\bigr|
        $$
        is maximum among all $d$-element scales, i.e. for any scale $B$ with $d$ elements, 
         ${\| \fa \|}^* \ge  {\| \fb \|}^*$. Moreover, if 
        $  {\| \fa \|}^* =  |\fa(t_0)|$, then
        $t_{0}^{-1}\in\zc$ is one generator of scale $A$, the only other being $- t_0^{-1}$. 
   \end{Thm} 
    This extends the discovery made by Ian Quinn \cite{Quinn}
   that $A$ is Maximally Even\footnote{ These sets were introduced by Clough and Myerson \cite{CM}, they can be seen as scales where the elements are as evenly spaced as possible on a number of given sites. See also \cite{CD, DK, Amiot}.} with $d$ elements if, and only if, $|\fa(d)|$ has maximum value among $d$ elements subsets; and for good reason, since the scales involved in the last theorem are affinely equivalent to some ME sets.
 It is also another illustration of his philosophy of looking at scales through their DFT. 
   We have elaborated on this in \cite{Amiot}, Online Supplementary III, wherein
   this theorem is proved, along with more complicated cases.\footnote{ Copies of the proof for scholarly purposes can be obtained from the author.}

\subsection{Generatedness and complementation}\ 

There is a rather suprising sequel to these theorems, already published in the case of
Maximally Even Sets \cite{Amiot}:
\begin{Thm}[Chopin's theorem for generated scales]\ 

\label{Chopin}
   If two generated scales $A,B$ are complementary, then some translate of one of them is 
   a subscale of the other. In particular, they have a common generator.
\end{Thm}
This stands for non trivial scales, that is to say when both the scale and its complement contain at least two notes. 
An example is (C D F G), i.e. $\{0,2,5,7\}$, and its eight notes complement in twelve-tone universe, see fig. \ref{complement}.

\dessin{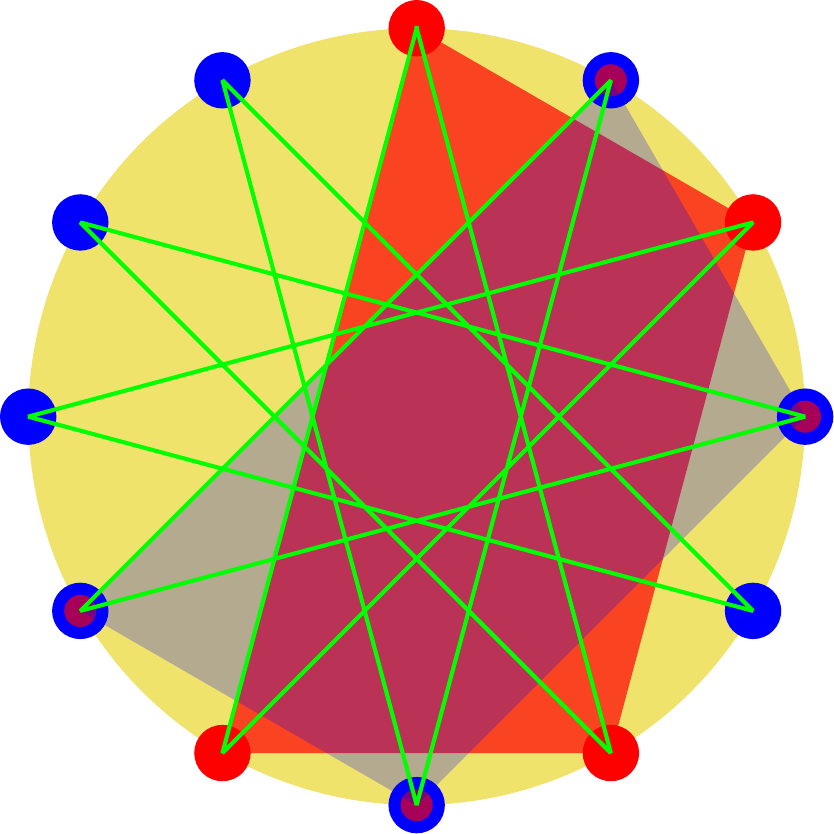}{7cm}{A generated scale and its generated complement have the same generators.}{complement}

\begin{Rem}\ 
   We named this Chopin's theorem, as it generalizes the construction used by Chopin in
opus 10, N. 5 etude in G flat major, where the right hand plays only black keys 
(e.g. a pentatonic scale, which is a Maximally Even Set\footnote{ See definition in the next paragraph.} generated by fifths) and the
left hand plays in several (mostly) major scales\footnote{ Three
major scales include the five black keys.} (another ME Set
generated by fifths), each of which includes the black keys, see fig. \ref{chopin}. 
In that situation\footnote{
 And even in a more general setting, with the subsets that Ian Quinn calls `prototypes',
which form  a class invariant by complementation \cite{Quinn}.} it has
certainly been observed before \cite{Amiot}, but as a statement on
general generated scales it is new, as far as we know, though it could
easily be derived from the Maximally Even case.
\end{Rem}

\dessin{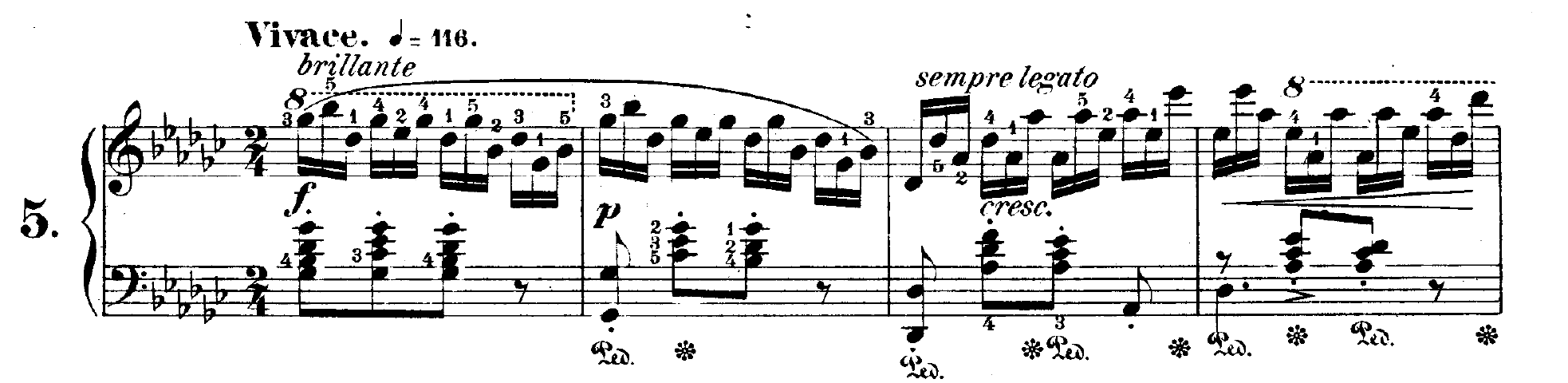}{15cm}{Chopin, Etude N. 5 opus 10.}{chopin}

\subsection{Non integer generation}
The present paper covered exhaustively the case of scales of the form
$$\{a + k\,f \pmod c, k=0\dots d-1\};$$ 
but in other music-theoretic models, generation from non integer steps is also common.
There are for instance the $J-$functions (see \cite{CD,CM, DK}):
$$
   J_\alpha(k) = \lfloor k\alpha \rfloor\pmod c,
   \quad\text{wherein usually } \alpha=\dfrac{c}d
$$
The symbol $\lfloor t\rfloor$ denotes the floor function, i.e. the greatest integer
lower than, or equal to $t$. Let us call $J_\alpha-$sets the sets of the form $\{J_\alpha(0); \dots J_\alpha(d-1)\}\subset \Z_c$.\footnote{Some degree of generalization is possible, see \cite{CD} for instance, but results merely in translations of the set.}
When $\alpha=c/d$, the $J_\alpha-$set is a maximally even set.
\medskip

Equally important are the $\mc P_x-$sets, made up with consecutive values of the maps $k\mapsto \mc P_x(k)= k\,x \pmod 1$, that express pythagorean-style scales (e.g. $x = \log_2(3/2)$) like the Well-Formed Scales \cite{Carey, CC}.  The question of different generators can be formulated thus:

\emph{If two sets of values of $J_\alpha, J_\beta$ (resp. $\mc P_x$, $\mc P_y$) 
are transpositionally equivalent, do we have necessarily $\alpha=\pm \beta$ (resp. $x=\pm y$)? Otherwise said, do such scales have exactly 2 generators and no more?}

We have already seen that the general answer is \emph{no}, for instance in Thm. \ref{polygon} when
$\alpha\in\N, d\,\alpha=c$ and the $J_\alpha-$set is a regular polygon (or similarly $d\,x\in\N$ for the $\mc P_x$ case).
Other cases are worth investigating.

Let us first consider the values of a $J$ function with a random multiplier, e.g. 
$J_{\alpha}(k)= \lfloor k\,\alpha \rfloor \pmod c$ with some $\alpha\in\mathbb R$. These values have been mostly
scrutinized when $\alpha = c/d$, for the generation of  ME-sets with $d$ elements in $\zc$.
\begin{Thm} \label{rationalCase}
A $J_{\alpha}-$set (up to translation) 
  does not characterize the pair $\pm\alpha$: there are infinitely many $\alpha$'s that
  give the same sequence.
\end{Thm}
Secondly, we state a result when the generator $x$, in a finite sequence of values of $\mc P_x$, is irrational:
\begin{Thm}\label{irrational}
  If the sets $\mc P_x^d =\{0, x, 2x, \dots (d-1)x\} \pmod 1$ and 
  $\mc P_y^d =\{0, y, 2y, \dots (d-1)y\} \pmod 1$
  are transpositionally equivalent, with $x$ irrational and $d>0$, then $x=\pm y\pmod 1$.
\end{Thm}
This has been implicitely known in the case of Well-Formed Scales in non tempered universes \cite{Carey}, but this theorem is more general.

Lastly, we characterize the generators of the `infinite scales' 
$\mc P_x^\infty = \{ n x\pmod1, n\in\Z\}$:
\begin{Thm} \label{infinite}
   Two infinite generated scales are equal up to translation, i.e. $\exists \tau, \mc P_x^\infty = \tau + \mc P_y^\infty$, if and only if they have the same generator up to a sign, i.e. $x =\pm y\pmod 1$, when  $x$ is irrational. In the case where $x$ is rational and $x=a/b$, there are $\Phi(b)$ different possible generators.
\end{Thm}
In other words, an infinite pythagorean scale has 2 or $\Phi(b)$ generators, according to whether the number of actually different notes is infinite or finite. The special (tritone) case of one generator already mentioned in Thm. \ref{notCoprime} also occurs for $x=1/2$.

\section{Proofs}
\subsection{Proof of Thm. \ref{2genCase}}
The following proof relies on the one crucial concept of (oriented) interval vector,\footnote{ A famous concept in music theory, see for instance \cite{Rahn}.} 
that is to say the multiplicities of all intervals inside a given scale. 
This is best seen by transforming $A, B$ into segments of the chromatic scale, by
way of affine transformations.
\begin{proof}
All computations are to be understood modulo $c$. The extreme cases $d=c-1$  and $d=c$ were
discussed before: any $f$ coprime with $c$ generates the whole $\zc$, hence
$\{f, 2f, \dots  (d-1)f\}$ is always equal to $\zc$ deprived of 0. Up to a change of starting point, we can generate with such an $f$ any subset with cardinality $d-1$.

Furthermore, a generator $f$ \textbf{not} coprime with $c$ would only generate 
(starting with 0, without loss of generality) a part of 
the strict subgroup  $f\zc\subset \zc$, hence a subset with stricly less than $c-1$ elements.

So we are left with the general case, $1<d<c-1$. Without loss of generality we take $b$ invertible modulo $c$ ($a$ and $b$ are interchangeable, and we assumed that one of them is invertible modulo $c$, i.e. coprime with $c$).\smallskip

Let $D=\{0, 1, 2, \dots  d-1\}$; assume $A=B+\tau$, as $A=a D$ and $B=b D$,  
then $D$ must be its own image -- $\varphi(D)=D$ -- under the following affine map: 
$$
\varphi:x\mapsto  b^{-1}(a\,x-\tau)  = b^{-1}a x -b^{-1}\tau = \lambda x + \mu.
$$
We now elucidate the different possible multiplicities of intervals between two elements of $D$.
\begin{Lemma}
   Let $1<d<c-1$, we define
   the oriented interval vector $V_D$ by $V_D(k)= \Card\{(x,y)\in D^2 ,  y-x = k\}$.
   Then  $V_{D}(k)<d-1\ \forall k=2\dots c-2$; more precisely,
   $$
     V_D = [V_D(0), V_D(1), V_D(2),\dots  V_D(c-1)] = [d, d-1,d-2,\dots, d-2,d-1]
   $$
   i.e. interval 1 and its opposite $c-1$ are the only ones with multiplicity $d-1$.
\end{Lemma}

For instance, with $c=12, d=8, k=5$ one computes
$V_D=[8, 7, 6, 5, 4, 4, 4, 4, 4, 5, 6, 7]$.
See a picture of such an interval vector of a chromatic cluster $D$ (fig. \ref{iv}).

\dessin{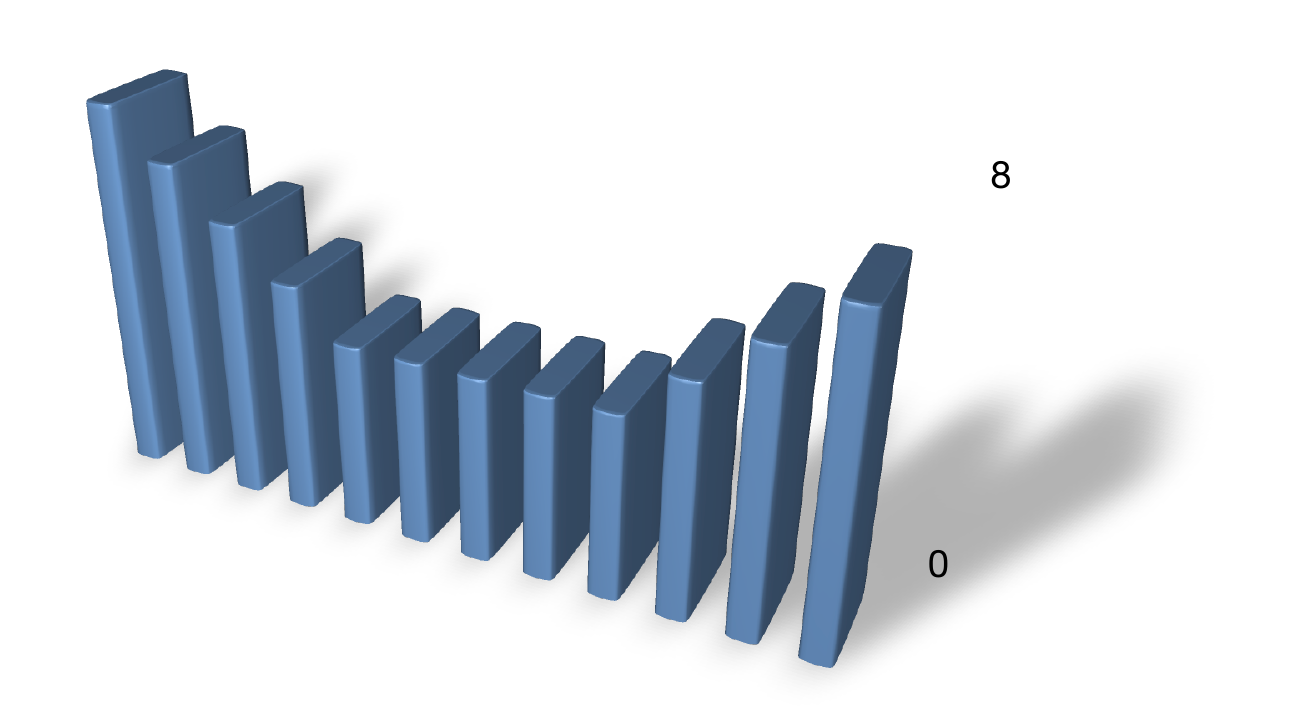}{10 cm}{Interval vector of (0 1 2 3 4 5 6 7) mod 12.}{iv}

\begin{proof}
Indeed, if we single out an interval $1<k\le c/2$ (this is general enough, owing to 
the obvious symmetry $V_D(c-k) = V_D(k)$),
the pairs $(i,j)$ of elements of $D$ which span exactly this interval $k$ come in two kinds (see fig.
\ref{caption}): either $i<j=i+k$ or the reverse, in which latter case the interval is in fact $j+c-i$, with
$j=i+k-c$:
$$
  (1,k+1) \dots  (d-k,d) \quad\text{when $k<d$, and, when $d+k>c$,}\quad
  (c+1-k,1) \dots  (d, d+k-c) ,
$$
which add up to
$\begin{cases}
     d-k & \text{pairs  for } k\le \inf(d,c-d)\\
     (d-k) + (d-c+k) = d-(c-d) & \text{pairs  for } c-d<k\le d\\
     d-c+k          & \text{pairs  for } k \ge \max(d, c-d)
\end{cases}.$ 

In all three cases, the multiplicity is $<d-1$, since $k, c-k$ and $c-d$ are all $>1$ by assumption.
\end{proof}

\dessin{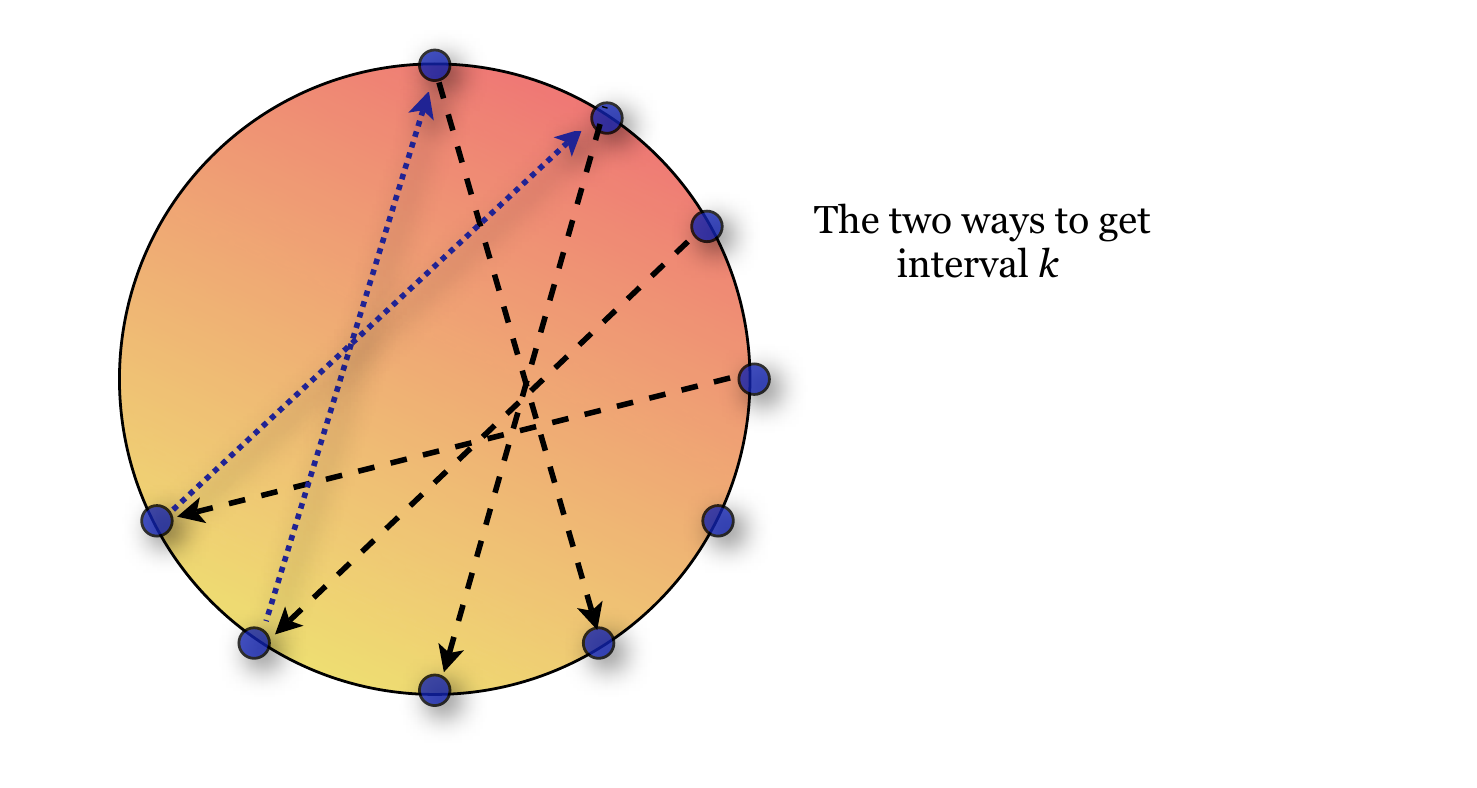}{12cm}{Double origin of one interval.}{caption}

As we will see independently below, if a generator is invertible so are all others.\footnote{ This could also be proved directly from $\varphi(D)=D$.} Hence 
$\lambda = b^{-1} a$ is invertible in $\zc$ and the map $\varphi:x\mapsto \lambda x + \mu$ above is one to one; it multiplies all intervals by $\lambda \mod c$:
$$
   \varphi(j) - \varphi(i) = (\lambda j + \mu) - (\lambda i + \mu) =\lambda. (j-i),
$$ 
which turns the interval vector 
$V_D$ into  $V_{\varphi(D)}$ wherein $V_D(\lambda i) = V_{\varphi(D)}(i)$: 
the same multiplicities occur, but for different intervals. 
This is a well known feature of affine transformations, that they permute the interval vectors.

Most notably, the {\em only}\footnote{ Because $\varphi$ is one to one.} 
two intervals with multiplicity $d-1$ in $V_{\varphi(D)}$ are $\lambda$ and $-\lambda$.
Hence, if $V_{\varphi(D)}=V_D$, the maximal multiplicity $d-1$ must appear in positions 1 
and $c-1$, which compels $\lambda$ to be equal to $\pm 1$.
Finally, as $\lambda=a\,b^{-1}\mod c$, we have indeed proved that $a=\pm b$, qed.
\end{proof}

\subsection{Proof of Thm. \ref{polygon}}

Now that we abandon the condition $\gcd(c,d)=1$, the affine maps $t\mapsto a\,t$,
for $a$ not coprime with $c$, are no longer one to one, so it is not clear how to
get back to chromatic clusters like in the proof of Thm. 1.

Let us begin with an interesting generalization of the reasoning in that proof. 
It generalizes a theorem about Well-Formed Scales in \cite{Carey}, and it applies to many musically pertinent objects, like the octatonic scale $\{0, 1, 3, 4, 6, 7, 9, 10\}$:
\begin{Lemma}
  A scale $A$ with $d$ notes is a reunion of regular polygons, each
   generated by the same $f$, if and only if  $V_A(f) = d$.
\end{Lemma}
\begin{proof}
    Let us enumerate the starting points in $A$ of interval $f$, i.e. all $a\in A$
   such that $\exists b\in A, a+f=b$. These are all elements of $A$.
   This means that $A+f$ is equal to $A$, hence $a\mapsto a+f$
   is a permutation of the set $A$.
   
   The order of this map is exactly the order of $f$ in group $\zc$, i.e.
   $m=c/\gcd(c,f)$. The orbits of this map are $m-$polygons, which proves
   the reverse implication. The direct sense is obvious, since each regular polygon
   provides exactly $m$ times the interval $f$, as 
   $(a+k f) - a = f \iff (k-1)\, f = 0$ in \zc, which implies that $k-1$ is a multiple of $m$
   and hence $a+k\,f = a+f$, the only successor of $a$ in the orbit.
\end{proof}
There are further extensions of this. First notice that a scale featuring an interval $f$ with
multiplicity $d-1$, when $d$ divides $c$, is not necessarily generated by this interval, as this scale can be built up, for instance, of several full orbits of $x\mapsto x+f$, plus one chunk of another orbit. This
is quite different from the $\gcd(c,d)=1$ case of Thm. 1. 
There are\footnote{ In \cite{Amiot}, Online Supplementary 3, 
we have studied which scales in general achieve the greatest
value of their maximum Fourier coefficient. It turns out that they exhibit this kind of
geometrical shape.} numerous musical occurences of such scales, for instance $\{2, 5, 8, 9, 11\}$ and $\{1, 4, 7, 10, 11\}$ in fig. \ref{Liszt}. 
This adds significance to Thm. \ref{irrational}.

\dessin{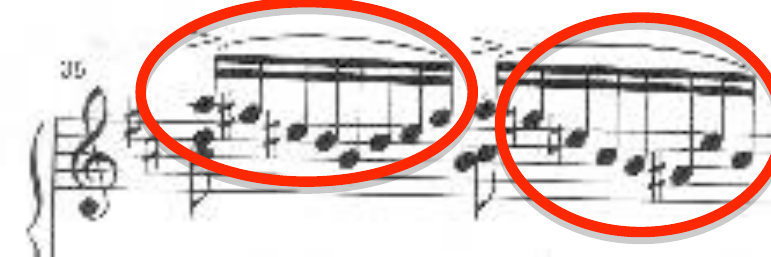}{8 cm}{Minor third with multiplicity 4 in 5 notes, in Liszt's sonata in B.}{Liszt}
Also,
\begin{Lemma}
   $A\subset\zc$, a scale with $d$ notes, is a regular polygon
    iff  $V_A(f) = d$ \textbf{for some divisor $\pmb f$ of $\pmb c$}.
\end{Lemma}
\begin{proof}
   Same as above, but there is only one orbit, i.e. a regular polygon, the orbit of some $a$,
 e.g.  $a +  f \zc$.
\end{proof}
In that case, a generator of the scale is the same thing as a generator of the difference group
$<A-A> = f \,\zc = \dfrac c m\zc$ \footnote{ 
This is the group generated by $A-A$, in all generality, cf. \cite{ToM}, 7.26.
It is illuminating to visualize this group as a kind of tangent space of $A$, akin to
the vector space associated with an affine structure.}.
This is the case pointed out by D. Clampitt\footnote{ Private communication.}.
As any subgroup of $\zc$ [with $d$ elements] must be cyclic, and has $\Phi(d)$ generators
(this being one definition of the totient function $\Phi$), this proves Thm. \ref{polygon}.

\subsection{Proof of Thm. \ref{notCoprime}}
The first case is obvious, when $c$ is even the only generator of $\{0, f=c/2\}$ is $c/2$. Case 3 was studied in Thm. \ref{polygon}.

There remains to be considered the case of a scale generated by some $f$ \textbf{not}
coprime with $c$, when that scale is not a regular polygon. 
For the end of this discussion, let $\gcd(f,c)=m>1$, and assume $d>1$ and
$0\in A$ (up to translation).

 We introduce $\fa(t)= \sum_{k\in A} e^{-2i\pi k t/c}$, the 
 Discrete Fourier Transform (DFT for short) of $A$. This is a secret weapon popularized in music theory by \cite{Quinn}, see \cite{Amiot} for details about the maths.
 
 When $A = \{f, 2f, \dots d\, f\}$ one gets from a simple trigonometric computation:
 \begin{Lemma}\label{trigo}
 $$|\fa(t)| = \begin{cases}\dfrac{|\sin (\pi\, d\, t\, f /c)|}{|\sin (\pi\, t\, f /c)|} & \text{or}\\
     \qquad d        & \text{ when }\sin (\pi t f /c)=0
    \end{cases} 
 $$
 Moreover, $|\fa(t)| \leq d$, and $|\fa(t)|=d \iff \sin (\pi t f /c)=0$.
\end{Lemma}
\begin{proof}
   The formula is derived from Euler's $2i\sin\theta = e^{i\theta} - e^{-i\theta}$
   and the computation of
   $$(e^{i\theta} - e^{-i\theta})\sum_{k=1}^d e^{-2 k i\theta}
     = e^{-i\theta} - e^{-(2d+1)i\theta}
     = e^{-i\theta}e^{-d\;i\theta} (e^{d\;i\theta} - e^{-d\;i\theta}).
   $$
Setting $\theta = \pi t f/c$ yields the first result. When $\theta\in\pi\Z$, all the exponentials are equal to 1 in the definition of $\fa$ and hence $\fa(t) = 1+1+1+\dots = d$. Conversely, this equality can only occur (from Minkowski's inequality) when all the exponentials point to the same direction, which only happens for $\theta\in\pi\Z$, i.e. when $\sin (\pi t f /c)=0$.
  
The remaining inequality comes from
    $|\sin(2\theta)| = 2 |\cos\theta \sin\theta| < 2 |\sin\theta|$ and by easy induction,
  $$|\sin(d\theta)| < d |\sin\theta|$$ for integer $d\ge 2$ and $0<\theta<\pi$ with 
 again $\theta = \pi t f/c \mod \pi$.
\end{proof}
 If $g$ is another generator, one can also write $A = a + \{g, 2g, \dots d\, g\} $, and again
  $|\fa(t)| = \dfrac{|\sin (\pi d t g /c)|}{|\sin (\pi t g /c)|}$
  (because $|\fa|$ does not change when $A$ is translated). 
  Notice that the maximum value $d$ of this DFT must occur for several values of $t$ ($t\neq 0 \mod c$), since we assumed that $f$ and $c$ are not coprime.
  Hence, as this quantity cannot reach maximum value $d$ unless both $\sin (\pi t g /c)$ and $\sin (\pi t f /c)$  are nil, $t\,g/c$ and $t\,f/c$ must get simultaneously integer values. Say $t_0>0$ is the smallest integer satisfying this, then $m = c/t_0$ divides both $f$ and $g$. Hence by maximality of $m$, we get
\begin{Lemma}  \label{twoGen}
   If $f, g$ are two generators of a same scale $A$, then $m=\gcd(c,f) = \gcd(c, g)$.
\end{Lemma}
NB: this lemma can also be reached algebraically, but it is not altogether trivial.

This proves also what we had advanced during the proof of Thm. 1, namely that generators of a same scale must have the same order. In particular, if one is invertible modulo $c$, then so is the other.\medskip

From there, one can divide $A$ by $m$ and assume without loss of generality that
$f'=f/m$ and $c'=c/m$ coprime. We are dealing now with a scale $A' = A/m$ in $\Z_{c'}$, generated by $f'$ and $g'=g/m$, both coprime with $c'$: then Thm. 1 gives two cases, either
$\Card A'=d < c'-1$ or  not. 
In the latter case, we have $\Phi(c')=\Phi(d+1)$ generators 
for an almost full, or full, aggregate; 
in the former, only two, like for the generic `major-like' scale. 

As for instance, $f' = \pm g' \pmod c' \iff f = \pm g \pmod c$, we have exhausted all possible cases when a generator is not coprime with $c$, and proved Thm. \ref{notCoprime}.

\subsection{Proof of Thm. \ref{Chopin}}
\begin{proof}
  We assume that scales $A, B$ are both generated, and that one is the complement of the other.
  
  First we study the more generic case when one generating interval is coprime with $c$; let for instance
  $A =\{ f, 2f, \dots d\,f\}$. Then $A'=\{(d+1)f, \dots c f=0\}$ is the complement  of $A$ in $\zc
  =\{f, 2f, \dots c f\}$ since the complete sequence of multiples of $f$ enumerates the whole of $\zc$.
  Hence $A'=B$, and $A, B$ share a generator: either $B-d\,f\subset A$, or the reverse.
  
  Lastly, let us assume that both $A$ and $B$ have no generator coprime with $c$. Then $A$
  (and $B$ likewise) is, up to translation, a subset of some subgroup of $\zc$, hence has at most $c/2$ elements. More precisely, if (say) $A=\{0,f,\dots (d-1)f\}$ then 
  $A$ is a subset of $m\zc$ where $m=\gcd(c,f)$, and $m\zc$ is a subgroup with $c' = c/m$ elements.
  Unless both cardinals are $c/2$, $A\cup B$ cannot be equal to the whole $\zc$.
  Thus the only remaining case whence $A$ can still be the complement of $B$ is when $c$ is even,
  and $A, B$ are complement halves (like for instance the two whole-tone scales $\{0, 2, 4, 6, 8, 10\}$ and $\{1, 3, 5, 7, 9, 11\}$),   which means that $\gcd(c,f)=2$, and 2 is a generator. 
  If so, then  $A=B+1$, which ends the proof.
\end{proof}

\subsection{Proof of Thms. \ref{rationalCase} and \ref{irrational}}
\begin{proof}
For Thm. 6, we note that a $d-$note scale produced by $J_\alpha$ does not change when $\alpha$ is augmented by a small enough quantity. Hence any $\alpha'>\alpha$ sufficiently close to $\alpha$ yields the same sequence. This can be seen on the following picture (fig. \ref{locConst}):

 \dessin{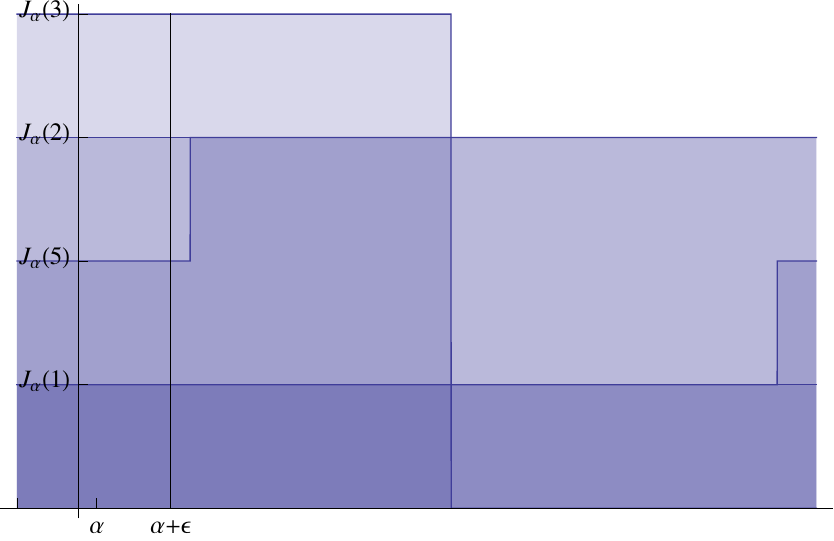}{9 cm}{The maps $\alpha\mapsto J_\alpha(k)$ are constant between $\alpha$ and $\alpha+\varepsilon$}{locConst}
 
\end{proof}

Now for the irrational generator and finite scale case (Thm. \ref{irrational}).
\begin{proof}
It follows the idea of the proof of Thm. 1, using the interval vector. This works because the affine map 
$\mc P_{x}$ is one to one again: since $x$ is irrational, we have the
\begin{Lemma}
  $\forall a,b\in\Z,\ a\, x \equiv b\,x\pmod 1 \iff a=b$.
\end{Lemma}

Consider now all possible intervals in $\mc P_x^d$, i.e. the $(i-j)x\pmod 1$ with $0\le i,j<d$.
By our hypothesis, these intervals occur with the same multiplicity in $\mc P_x^d$ and $\mc P_y^d$.
Let us have a closer look at these intervals (computed modulo 1), noticing first that

\begin{itemize}
  \item
  There are $d$ different intervals from 0 to $k\,x$, with $k=0$ to $d-1$. They are distinct
  because $x$ is irrational, as just mentioned. Their set is $\mc I_0 =(0,x,2x\dots (d-1)x)$.
  \item
  From $x$ to $x, 2x, 3x, \dots (d-1)x$ \emph{and 0}, there are $d-1$ intervals common
  with $I_0$, and a new one, $0-x=-x$. It is new because $x$ is still irrational. For the record,
  their set is $\mc I_x =(0,x,2x\dots (d-2)x, \pmb {-x})$.
  \item 
  From $2x$ to the others, $d-1$ intervals are common with $\mc I_x$ and only 
  $d-2$ are common with the $\mc I_0$.
  \item
  Similarly for $3x, 4x\dots$ until
  \item
  Finally, we compute the intervals from $(d-1)x$ to $0, x,  \dots,  (d-2)x, (d-1)x$.
  One gets  $\mc I_{(d-1)x} =(0, -x, -2x\dots -(d-1)x)$.  
\end{itemize}
The following table, not unrelated to fig. \ref{iv}, will make clear the 
values and coincidences of the different possible intervals (see fig. \ref{tabInt}):

 \dessin{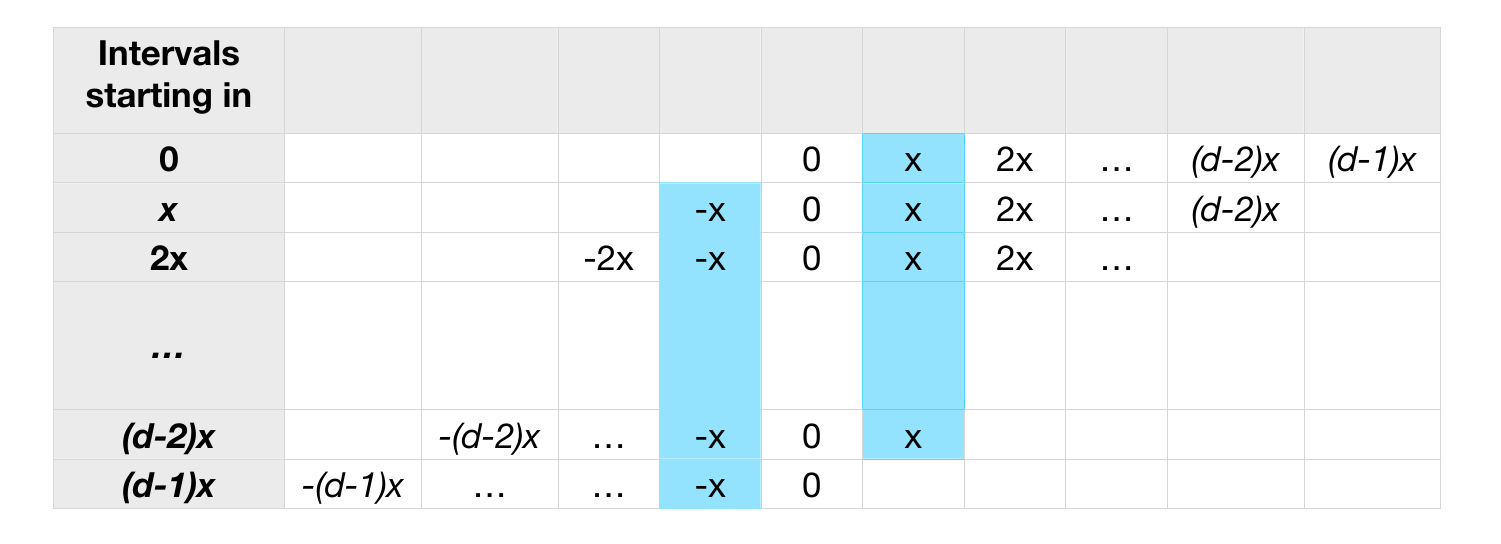}{16 cm}{The different intervals from each starting point}{tabInt}
 
So only two intervals (barring 0) occur $d-1$ times in $\mc P_x^d$ (resp. $\mc P_y^d$),
namely $x$ and $-x$. Hence $x=\pm y$, qed.
\end{proof}

\subsection{Proof of Thm. \ref{infinite}}
Notice that $\mc P_x^\infty = \{ k x\pmod1, k\in\Z\}$ is a subgroup of the circle (or one-dimensional torus) $\R/\Z$, quotient group of the subgroup of $\R$ generated by 1 and $x$. All computations are to be understood modulo 1.
If we consider a translated version $\mc S = a + \mc P_x^\infty$, then the group can be retrieved by a simple difference:
$$
   \mc P_x^\infty = \mc S -\mc S = \{ s-s', (s, s')\in\mc S^2\}
$$
So the statement of the theorem can be simplified, without loss of generality, as 
``if $\mc P_x^\infty =\mc P_y^\infty $ then $x=\pm y\pmod 1$'' (and similarly in the finite case).
\begin{proof}
We must distinguish the two cases, whether $x$ is rational or not.
\begin{itemize}
  \item
  The case $x$ rational is characterized by the finitude of the scale. Namely, when $x = a/b$ with $a, b$ coprime integers (we will assume $b>0$), then $\mc P_x^\infty$ is the group generated by $1/b$: one inclusion is clear, the other one stems from Bezout relation: there exists some combination $a u + b v = 1$ with $u,v$ integers, and hence
  $$1/b = u \, a/b + v = u\, a/b  \pmod 1= \underbrace{a/b+\dots a/b}_{u \text{ times}}$$ is an element of $\mc P_x^\infty $. As $1/b\in \mc P_x^\infty$, it contains the subgroup generated by $1/b$, and finally these two subgroups are equal.
  
  The subgroup $<1/b>\pmod 1$ is cyclic with $b$ elements, hence it has $\Phi(b)$ elements, which concludes this case of the theorem.\footnote{ The different generators are the $k/b$ where $0<k<b$ is coprime with $b$.}
  \item
  Now assume $x$ irrational and $\mc P_x^\infty = \mc P_y^\infty$. 
  An element of $\mc P_x^\infty$ can be written as $a x \pmod 1$, with $a\in\Z$. Since $y\in\mc P_x^\infty$ then $y = a x\pmod 1$ for some $a$. Similarly, $x = b y$ for some $b$. Hence
  $$
    x = a b x \pmod 1 \quad\text{ that is to say in $\Z$, }\quad
    x = a b x + c
  $$
  where $a,b,c$ are integers.
  This is where we use the irrationality of $x$: $(1 - a b) x = c$ implies that $ab=1$ and $c=0$.
  
  Hence $a = \pm 1$, i.e. $x=\pm y \pmod 1$, which proves the last case of the theorem.
\end{itemize}
\end{proof}
\begin{Rem}
   The argument about retrieving the group from the (possibly translated) scale applies also if the scale is just a semi-group, e.g. $\mc P_x^{+\infty} = \{ k x\pmod1, k\in\Z, k\ge 0\}$, which is perhaps a less wild generalization of the usual pythagorean scale. So the theorem still holds for the half-infinite scales.
\end{Rem}
\section*{Conclusion}
Apart from the seminal case of Major Scale-like generated scales, it appears that
many scales can be generated in more than two ways. This is also true for
more complicated modes of `generation'. 

Other simple' sequences appear to share many different generation modes. It is true for instance of geometric sequences, like the \textsl{powers} of 3, 11, 19 or 27 modulo 32 which
generate the same 8-note scale in $\Z_{32}$, namely 
$\{1, 3, 9, 11, 17, 19, 25, 27\}$ -- geometric progressions being interestingly dissimilar 
in that respect from arithmetic progressions.\footnote{ Such geometric sequences
occur in Auto-Similar Melodies \cite{SSM}, like the famous initial motive in Beethoven's
Fifth Symphony, autosimilar under ratio 3.}

I hope the above discussion will shed some light on the mechanics of scale construction.
I thank David Clampitt for fruitful discussions on the subject, 
and Ian Quinn whose ground-breaking work edged me on to explore the subject in depth, and the readers whose fruitful comments helped me to hone both my english and the readability of this paper.


\begin{thebibliography}{13}

\bibitem{Amiot} Amiot, E., nov. 2007, {\em {David Lewin and Maximally Even Sets}}, 
in:  {\em Journal of Mathematics and Music}, Taylor and Francis 
\textbf{1}(3):152-172.

\bibitem{SSM} Amiot, E., nov. 2008, {\em Self Similar Melodies}, in:  {\em Journal of Mathematics and Music}, Taylor and Francis \textbf{2}(3):157-180.

\bibitem{Carey} Carey, N., 1998, {\em Distribution Modulo 1 and Musical Scales},
PhD thesis, University of Rochester. Available online.


\bibitem{CC} Carey, N., Clampitt, D., 1989, Aspects of Well Formed Scales,
{\itshape  Music Theory Spectrum}, \textbf{11}(2),187-206.

\bibitem{CD} Clough, J., Douthett, J., 1991,  Maximally Even Sets, 
{\itshape  Journal of Music Theory}, \textbf{35:}93-173.

\bibitem{CM} Clough, J., Myerson, G., 1985, Variety and Multiplicity in Diatonic Systems,
{\itshape Journal of Music Theory}, \textbf{29:}249-70.

\bibitem{DK} Douthett, J., Krantz, R., 2007, Maximally even sets and configurations: 
common threads in mathematics, physics, and music, {\itshape 
Journal of Combinatorial Optimization}, Springer.

Online: http://www.springerlink.com/content/g1228n7t44570442

\bibitem{Lewin59} Lewin, D., 1959,
{\em Re: Intervalic Relations between two collections of notes}, 
in:  {\em Journal of Music Theory},  3:298-301.

\bibitem{GIS} Lewin, D., 1987, {\em Generalized Musical Intervals and Transformations},
 New Haven:Yale University Press.

\bibitem{Quinn} Quinn, I., 2004, {\em A Unified Theory of Chord Quality in 
Equal Temperaments},  Ph.D. dissertation, Eastman School of Music.

\bibitem{Rahn} Rahn, D., 1980, {\em Basic Atonal Theory}, Longman, New York.


\bibitem{ToM} Mazzola, G., et alii, 2002, {\em Topos of Music}, Birkh\"auser.

\end{thebibliography}
\end{document}